\documentclass[twoside, paper=a4, fontsize=11pt]{article} 
\usepackage[colorlinks=true,linkcolor=blue,citecolor=blue,urlcolor=blue]{hyperref}
\usepackage[T1]{fontenc}
\usepackage[utf8]{inputenc}
\usepackage[english]{babel}

\usepackage[bbgreekl]{mathbbol}
\usepackage{amsmath,amssymb,amsfonts,amsthm}
\usepackage{mathrsfs}

\usepackage[charter]{mathdesign}

\usepackage[pdftex]{graphicx}	
\usepackage{url}

\usepackage{titlesec}
\titleformat{\section}[hang]{\Large \normalfont \filcenter \scshape}{\thesection}{10pt}{}
\titleformat{\subsection}[hang]{\normalfont \large \bfseries}{\thesubsection}{6pt}{}

\titlespacing*{\section}{0pt}{14pt}{3pt}
\titlespacing*{\subsection}{0pt}{14pt}{3pt}
\titlespacing*{\paragraph}{0pt}{7pt}{8pt}

\usepackage{enumerate}
\usepackage{enumitem}

\usepackage{fancyhdr}
\usepackage[titletoc]{appendix}
\AtBeginDocument{}
\appendixtocon

\usepackage[numbers,comma,sort]{natbib}
\usepackage{amsmath}
\usepackage{caption}	

\usepackage[affil-it]{authblk}

\numberwithin{equation}{section}

\newtheorem{theorem}{Theorem}[section]
\newtheorem{proposition}{Proposition}[section]
\newtheorem{lemma}{Lemma}[section]

\newtheorem{example}[lemma]{Example}
\newtheorem{definition}[lemma]{Definition}

\newtheorem{remark}[lemma]{Remark}

\begin{document}

\newcommand{\Hc}{\mathcal{H}}
\newcommand{\A}{\mathcal{A}}
\newcommand{\C}{\mathbb{C}}
\newcommand{\R}{\mathbb{R}}
\newcommand{\Q}{\mathbb{Q}}
\newcommand{\Z}{\mathbb{Z}}
\newcommand{\N}{\mathbb{N}}
\newcommand{\levy}{\mathbb{X}^H}
\newcommand{\sheet}{\mathbb{W}^{\mathbf{H}}}
\newcommand{\BB}{\boldsymbol{B}}
\newcommand{\alphar}{\texttt{$\boldsymbol{\alpha}$}}
\newcommand{\sigmar}{\texttt{\large $\boldsymbol{\sigma}$}}
\newcommand{\qA}{q_{\underline{\mathcal{A}}}}
\newcommand{\EE}{\mathbb{E}}
\newcommand{\PP}{\mathbb{P}}
\newcommand{\VV}{\mathbb{V}\text{ar}}
\newcommand{\h}{\normalfont{\textit{\textbf{h}}}}
\newcommand{\dH}{\normalfont{\text{dim}}_H}
\newcommand{\dd}{\normalfont{\text{d}}}
\newcommand{\F}{\mathbf{F}}
\newcommand{\x}{\mathtt{\mathbf{x}}}
\newcommand{\psiu}{\Psi_h^{(u)}}
\newcommand{\psil}{\Psi_h^{(\ell)}}
\newcommand{\psiut}{\widetilde{\Psi}_h^{(u)}}
\newcommand{\psilt}{\widetilde{\Psi}_h^{(\ell)}}
\newcommand{\etau}{\eta^{(h,u)}}
\newcommand{\etal}{\eta^{(h,\ell)}}

\title{\vspace{-1cm} \normalfont \Large \uppercase {Increment stationarity of $L^2$-indexed stochastic processes: spectral representation and characterization}} 

\author{Alexandre Richard}
\affil{\small TOSCA team -- INRIA Sophia-Antipolis\\  2004 route des Lucioles, F-06902 Sophia-Antipolis Cedex, France\\ e-mail: \href{mailto:alexandre.richard@inria.fr}{alexandre.richard@inria.fr}}

\date{}

\maketitle

\begin{abstract}
We are interested in the increment stationarity property for $L^2$-indexed stochastic processes, which is a fairly general concern since many random fields can be interpreted as the restriction of a more generally defined $L^2$-indexed process. We first give a spectral representation theorem in the sense of \citet{Ito54}, and see potential applications on random fields, in particular on the $L^2$-indexed extension of the fractional Brownian motion. Then we prove that this latter process is characterized by its increment stationarity and self-similarity properties, as in the one-dimensional case.
\end{abstract}



\section{Introduction}

It is known since the works of Ito \cite{Ito54} and Yaglom \cite{Yaglom57}, that if a (multiparameter) stochastic processes $X$ is increment-stationary in the sense that for any $s,s',t,t'$ and $h \in \R^d$:
\begin{equation*}
\EE\left((X_{t+h}-X_{s+h}) (X_{t'+h}-X_{s'+h})\right) = \EE\left((X_t-X_s) (X_{t'}-X_{s'})\right)\ ,
\end{equation*}
then $X$ admits a spectral representation, understood as: there exist a random measure $M$ on $\R^d$, with control measure $m$, and an uncorrelated random vector $Y \in L^2(\Omega;\R^d)$ such that:
\begin{equation*}
\forall t\in \R^d,\quad X_t = \int_{\R^d} \left(e^{i\langle t, x\rangle}-1\right)\ M(\dd x) + \langle t, Y\rangle\ .
\end{equation*}
Such representations have important applications in the study of sample path properties of stochastic processes (see \cite{Monrad, Talagrand}, to cite but a few). However, some processes that appear now frequently in the literature (for instance in the domain of stochastic partial differential equations \cite{Dalangminicourse,BalanJolisQS}) possess a different type of stationarity. This is the case of the Brownian sheet (the random field whose distributional derivative is the white noise on $\R^d$), and more generally of the fractional Brownian sheet (see Example \ref{ex:111}). Let us write this increment stationarity property in $\R^2$: let $\mathbb{W}$ be a fractional Brownian sheet, and define its increments by: $\Delta_{[u,v]} \mathbb{W} = \mathbb{W}_{v} - \mathbb{W}_{(u_1,v_2)} - \mathbb{W}_{(v_1,u_2)} + \mathbb{W}_{u}$ for any $u=(u_1,u_2) \preccurlyeq v=(v_1,v_2)\in \R^2$. Then for any $u\preccurlyeq v,\ u'\preccurlyeq v'$ and $h\in \R^2$,
\begin{equation*}
\EE\left(\Delta_{[u+h,v+h]} \mathbb{W}\ \Delta_{[u'+h,v'+h]} \mathbb{W}\right) = \EE\left(\Delta_{[u,v]} \mathbb{W}\ \Delta_{[u',v']} \mathbb{W}\right)\ .
\end{equation*}
This led \citet{BasseOConnor} to propose another spectral representation theorem for these processes, which permitted the construction of multiparameter stochastic integrals against these processes in the sense of Walsh.

Using a different technique, we obtain a similar result in Section \ref{sec:2}, for a larger class of processes. Our Theorem \ref{th:00} states that any random field $\{X(f),\ f\in L^2(T,m)\}$, where $(T,m)$ is any measure space such that $L^2(T,m)$ is separable, which has second moments and satisfies the following increment-stationarity property: $\forall f,f',g,g',h\in L^2(T,m),$
\begin{equation*}
\EE\left[(X(f+h)-X(g+h))\ (X(f'+h)-X(g'+h))\right] = \EE\left[(X(f)-X(g))\ (X(f')-X(g'))\right]\ ,
\end{equation*}
admits a spectral representation.
We explain in paragraph \ref{subsec:def} why this property covers many random fields, and how such random fields appear as the restriction of some $L^2(T,m)$-indexed process. In particular, all the known multiparameter extensions of the fractional Brownian motion are part of this class of processes. The counterpart for having such level of generality is that in some cases the resulting sprectral representation is either degenerate, or expressed in a too abstract setting for potential applications. However there are examples where the theorem permits to deduce sample path properties of multiparameter processes \cite{Richard2}. In all this section, the prototypical example of a process to which our spectral representation theorem applies is the $L^2(T,m)$-indexed fractional Brownian motion (defined in \cite{Richard} as an extension of the set-indexed fractional Brownian motion \cite{ehem}).

Hence in Section \ref{sec:3} of this paper, we focus on the $L^2(T,m)$-indexed fBm. For any $H\in (0,1/2]$, this real-valued centred Gaussian process has a covariance given by:
\begin{equation}\label{eq:covariance}
\EE\left(\BB^H(f)\ \BB^H(g)\right) = \frac{1}{2}\left(m\left(f^2\right)^{2H} + m\left(g^2\right)^{2H} - m\left((f-g)^2\right)^{2H}\right)\ ,\quad \forall f,g\in L^2(T,m)
\end{equation}
where $m(\cdot)$ denotes the linear functional $\int_T \cdot\ \dd m$ of $L^2(T,m)$. It encompasses most of the different known extensions of the fractional Brownian motion. We characterize the $L^2(T,m)$-indexed fractional Brownian motions in terms of self-similarity and increment-stationarity properties. Let us recall that the fractional Brownian motion of Hurst parameter $H\in (0,1)$ is the only (up to normalization of its variance) Gaussian process on $\R$ that has stationary increments and self-similarity of order $H$. In the multiparameter setting, things become more complicated as there are several possible definitions of increment stationarity as well as self-similarity. For instance, the Lévy fractional Brownian motion of parameter $H$, whose covariance is given by:
\begin{equation*}
\EE\left(\mathbb{X}^H_s \mathbb{X}^H_t\right) = \frac{1}{2}\left(\|s\|^{2H} + \|t\|^{2H} -\|s-t\|^{2H}\right)\ ,\quad s,t\in \R^d,
\end{equation*}
is self-similar of order $H$ and has a strong increment stationarity property on $\R^d$, i.e. against translations and rotations in $\R^d$:
\begin{equation*}
\forall g\in \mathcal{G}(\R^d),\quad \{\mathbb{X}_{g(t)}-\mathbb{X}_{g(0)},\ t\in \R^d\}\overset{(d)}{=} \{\mathbb{X}_t,\ t\in \R^d\} \ ,
\end{equation*}
where $\mathcal{G}(\R^d)$ is the group of rigid motions of $\R^d$. Reciprocically, it is the only Gaussian process having these properties, up to normalization of its variance \cite[p.393]{Samorodnitsky}. There is no such simple characterization for the fractional Brownian sheet (see the review \cite{ehemMpfBm}). We extend the notions of self-similarity and increment stationarity introduced in \cite{ehem,ehemMpfBm}, and give two characterizations of the $L^2$-fBm, depending on the definition of self-similarity and increment stationarity that are chosen for $L^2$-indexed processes.

\section{Spectral representation of $L^2$-stationary processes}\label{sec:2}

\subsection{Preliminaries}

A \emph{Gel'fand triple} consists in a Hilbert space $H$ and a larger space $E$ such that $H$ is densely and continuously embedded into $E$. We further assume here that $E$ is a Banach space. We shall denote by $E^*$ the topological dual of $E$, thus the inclusion $E^*\subset H^*$ leads to write $E^*\subset H\subset E$ for any Gel'fand triple. To continue with notations, we will use the duality bracket symbol $\langle \xi,x\rangle$, for any $\xi\in E^*$ and $x\in E$.\\
In general, the embedding between $H$ and $E$ is continuous. We will need it to be Hilbert-Schmidt for an extension of Bochner's theorem to be valid. This is the content of the following lemma, proved in \cite{Richard2} (actually with slightly stronger conclusions than written here).

\begin{lemma}\label{lem:HS}
Let $H$ be a separable Hilbert space. There is a separable Hilbert space $(E,\|\cdot\|)$ such that $E^*\subset H\subset E$ is a Gel'fand triple and the embedding $H\subset E$ is Hilbert-Schmidt.
\end{lemma}

Some Hilbert spaces associated to a kernel (as for instance a covariance) will be particularly useful.
\begin{definition}[Reproducing Kernel Hilbert Space]\label{def:RKHS}
Let $(T,m)$ be a separable and complete metric space and $C$ a continuous covariance function on $T\times T$. $C$ determines a unique Hilbert space $H(C)$ satisfying the following properties: \emph{i)} $H(C)$ is a space of functions mapping $T$ to $\R$; \emph{ii)} for all $t\in T$, $C(t,\cdot) \in H(C)$; \emph{iii)} for all $t\in T$, $\forall f\in H(C)$, $\left(f,C(t,\cdot)\right)_{H(C)} = f(t)$ .
\end{definition}

\noindent $H(C)$ is spanned by the set of mappings $\{C(t,\cdot),\ t\in T\}$, thus one can extract a basis of $H(C)$ of the form $\{C(t_n,\cdot),\ t_n\in T\}$. In the sequel, $(T,m)$ is always assumed to be a separable and complete metric space, which ensures that $L^2(T,m)$ is separable.\\

As mentioned in the introduction, spectral representations involve random measures. We provide below a formal definition of such objects.
\begin{definition}
Let $m$ be a finite measure on the Borel sets of a topological space $T$, which are denoted by $\mathcal{B}(T)$. Let $(\Omega,\mathcal{F},\PP)$ be a probability space.
A complex-valued \emph{random measure} on $\mathcal{B}(T)$ with control measure $m$ is a measurable mapping $M:\mathcal{B}(T)\rightarrow L^2_\C(\Omega)$ satisfying:
\begin{enumerate}[label=(\roman*)]
\item zero mean: $\EE\left(M(A)\right) = 0$ for any $A\in \mathcal{B}(T)$;
\item finite additivity: $M(A\cup B) = M(A)+M(B)\ a.s.$ for any disjoint $A,B\in \mathcal{B}(T)$;
\item covariance: $\EE\left(M(A)\ \overline{M(B)}\right) = m(A\cap B)$ for any $A,B\in \mathcal{B}(T)$.
\end{enumerate}
Since we will only encounter symmetric (random) measures on vector spaces, we readily assume that for any $A\in \mathcal{B}(T)$, $M(A) = \overline{M(-A)}\ a.s.$
\end{definition}
Stochastic integrals with respect to a random measure can be defined for deterministic integrands. As usual, the first step is to define it for elementary functions via the relation $\int \mathbf{1}_A \dd M = M(A)$, then extending it to simple functions. This establishes a linear isometry between $L^2_\C(m)$ and $L^2_\C(\Omega)$ (in the sequel we drop the $\C$ indexing, unless it needs to be specified). This isometry extends to the entire space $L^2(m)$.

\subsection{Definitions of increment-stationarity and examples}\label{subsec:def}

In this paragraph, we precise the terminology related to stationarity that we use in this article. Note that our main result concerns $L^2$-indexed stochastic processes, and since most random fields of interest are neither indexed by an infinite-dimensional vector space, nor even a vector space, our goal here is also to explain why this setting is interesting nonetheless.

For a given second-order\footnote{i.e. having finite second moments.} $T$-indexed random field $X$ with covariance $C$, we will consider the following objects: if there exist an $H(C)$-valued mapping $\mathtt{f}:t\in T \mapsto \mathtt{f}_t \in H(C)$ and an $H(C)$-indexed process $\widehat{X}$ such that $X_t = \widehat{X}(\mathtt{f}_t)$ for any $t\in T$, then we say that $X$ is \emph{compatible with $H$-indexing}. In case there exist a set-valued mapping $A:t\in T \mapsto A_t \in \mathcal{B}(T)$ and an isometry mapping $\mathtt{f}_t$ to $\mathbf{1}_{A_t}$ in some $L^2(T,m)$ space, we say that $X$ is \emph{compatible with set-indexing}. In the examples we will provide, the existence of $\widehat{X}$ is facilitated by the fact that the law of $X_t$ can be expressed in terms of $m(A_t)$, for some measure $m$.

\begin{example}\label{ex:indexings}
We begin with a few examples of interesting set-valued mappings.
\begin{enumerate}[parsep=0.0cm,topsep=0.1cm,itemindent=0.5cm,leftmargin=0.0cm]
\item The simplest example that comes to mind is the collection of rectangles of $\R^d$: $A_t = [0,t]$ and $m$ is the Lebesgue measure.
\item\label{item:Levy} There is a mapping $A$ and a measure $m_d$ on $\R^d$ such that $m_d(A_t\bigtriangleup A_s) = \|t-s\|$ for any $s,t\in \R^d$, where $\|\cdot\|$ is the Euclidean norm and $\bigtriangleup$ is the symmetric difference of sets.  Roughly, $A_t$ is the set of all hyperplanes that separates $0$ and $t$. This construction is fully described in \cite[Chap. 4]{Lifshits} or \cite[p.401]{Samorodnitsky}.
\item\label{item:sheet} A similar construction due to Takenaka (see also \cite[p.402-403]{Samorodnitsky}) gives the existence for $H\in(0,1/2]$ of a measure $m_d^H$ and a set-valued mapping $A$ such that $m_d^H(A_t\bigtriangleup A_s) = \|t-s\|^{2H}, \forall t,s\in \R^d$. Identically for a vector $\mathbf{H} = (H_1,\dots,H_d)\in (0,1/2]^d$, one can construct, by tensorization of one-dimensional measures, a new measure $m_d^{\mathbf{H}}$ and a set-valued mapping $A$ such that $m_d^{\mathbf{H}}(A_t\bigtriangleup A_s) = \prod_{k=1}^d \|t_k-s_k\|^{2H_k}$. 
\end{enumerate}
\end{example}

Let $\text{dom}\ \widehat{X}$ be the domain of definition of $\widehat{X}$.

\begin{definition}\label{def:WSIS}
We will say that a centred random field $X$ indexed by $(T,m)$ is \emph{wide-sense increment-stationary} if the following set of assumptions holds:
\begin{enumerate}[label=(\roman*),parsep=0.1cm]
\item $X$ is compatible with $L^2$-indexing for the mapping $\mathtt{f}$ ($\widehat{X}(\mathtt{f}_t) = X_t$ for any $t\in T$) and $\text{dom}\ \widehat{X}$ is a subvector space of $H(C)$;
\item $\widehat{X}$ is $L^2$-increment stationary, i.e. it has finite second moments at any point and it satisfies, for any $f_1,f_2,g_1,g_2$ and $h\in \text{dom}\ \widehat{X}$:
\begin{equation*}
\EE\left((\widehat{X}(f_1+h)-\widehat{X}(f_2+h)) (\widehat{X}(g_1+h)-\widehat{X}(g_2+h))\right) = \EE\left((\widehat{X}(f_1)-\widehat{X}(f_2)) (\widehat{X}(g_1)-\widehat{X}(g_2))\right) \ .
\end{equation*}
\end{enumerate}
\end{definition}

\

Let us remark that the existence of $\widehat{X}$ is close to the notion of ``model'' described in \cite{Lifshits}, although it is slightly less demanding. The choice of this type of stationarity for $\widehat{X}$ is motivated by the spectral representation theorem of the next section. Note also that the ``$L^2$'' term in ``$L^2$-increment stationary'' refers to the $L^2(T,m)$ indexing and not to the existence of the second moments of the process. \\

We present now a few wide-sense increment-stationary processes based on the examples of measure spaces given above.
\begin{example}\label{ex:111}
\begin{enumerate}[parsep=0.0cm,itemindent=0.5cm,leftmargin=0.0cm]
\item For any fixed $H\in (0,1)$ ($H=1/2$ corresponds to the Brownian case), there is a centred Gaussian process indexed by $\R^d$ which has the following increments:
\begin{align*}
\EE\left((\levy_t-\levy_s)^2\right) = \|t-s\|^{2H}\ .
\end{align*}
This process is called Lévy (fractional) Brownian motion and has the \emph{simple increment stationarity} property: $\EE\left((\levy_{t+h}-\levy_{s+h})(\levy_{t'+h}-\levy_{s'+h})\right)=\EE\left((\levy_{t}-\levy_{s})(\levy_{t'}-\levy_{s'})\right)$ for any $s,s',t,t',h\in \R^d$. Besides, the Euclidean space is compatible with set-indexing (see Example \ref{ex:indexings} point \ref{item:Levy} for the definition of $A_t$ and $m_d$) and the $L^2(\R^d,m_d)$-indexed Gaussian process defined by:
\begin{equation*}
\EE\left(\widehat{\mathbb{X}}^H(f)\ \widehat{\mathbb{X}}^H(g)\right) =\frac{1}{2}\left(m_d(f^2)^{2H} + m_d(g^2)^{2H} - m_d((f-g)^2)^{2H}\right)
\end{equation*}
is well-defined for $H\leq 1/2$ (see \cite{ehem}) and for any $t\in \R^d$, $\widehat{\mathbb{X}}^H(\mathbf{1}_{A_t}) = \levy_t$.
\item The fractional Brownian sheet $\sheet$ of Hurst parameter $\mathbf{H} = (H_1,\dots,H_d)\in (0,1)^d$ is the centred Gaussian process with covariance: let $\mathbf{t} = (t_1,\dots, t_d)\in \R_+^d,\ \mathbf{s} = (s_1,\dots, s_d)\in \R_+^d$,
\begin{align*}
\EE\left(\sheet_\mathbf{t}\ \sheet_\mathbf{s}\right) &= 2^{-d} \prod_{k=1}^d \left(|t_k|^{2H_k} + |s_k|^{2H_k} - |t_k-s_k|^{2H_k}\right) \\
&= R^{\otimes d}_\mathbf{H}(\mathbf{1}_{[0,\mathbf{t}]}, \mathbf{1}_{[0,\mathbf{s}]}) \ .
\end{align*}
$R^{\otimes d}_\mathbf{H}$ is a notation that holds when $\mathbf{H}\in (0,1/2]^d$. Indeed, for $H_k\in (0,1/2]$ and $f,g\in L^2(\R_+,\lambda_1)$, $R_{H_k}(f,g) = 1/2(\|f\|^{4 H_k}+\|g\|^{4 H_k}-\|f-g\|^{4 H_k})$ is a particular case of (\ref{eq:covariance}) ($\lambda_d$ will denote the $d$-dimensional Lebesgue measure) and $R_{H_k}(\mathbf{1}_{[0,t_k]},\mathbf{1}_{[0,s_k]})$ appears in the above product. The tensor product of such covariances yields a covariance on $\bigotimes_{k=1}^d L^2(\R_+,\lambda_1)$ which is isometric to $L^2(\R_+^d,\lambda_d)$, thus $R^{\otimes d}_\mathbf{H} = \bigotimes_{k=1}^d R_{H_k}$. Let $\widehat{\mathbb{W}}^\mathbf{H}$ be the $L^2(\R_+^d)$-indexed Gaussian process with covariance $R^{\otimes d}_\mathbf{H}$.\\
$\widehat{\mathbb{W}}^\mathbf{H}$ is $L^2$-increment stationary: this follows from the \emph{sheet increment stationarity} property of $\sheet$. This property is the main object of study in \cite{BasseOConnor} and is expressed as follows: for any $s\preccurlyeq t$, $s'\preccurlyeq t'$ and $u\in \R^d$,
\begin{align*}
\EE\left(\Delta \sheet ([s+u,t+u])\ \Delta \sheet ([s'+u,t'+u])  \right) = \EE\left(\Delta \sheet ([s,t])\ \Delta \sheet ([s',t'])  \right)
\end{align*}
where $\Delta \sheet$ is the process obtained by the inclusion-exclusion formula. That is, for $\mathbf{s}\preccurlyeq \mathbf{t}$, $\Delta \sheet ([\mathbf{s},\mathbf{t}]) = \sum_{\boldsymbol{\epsilon}\in \{0,1\}^d} (-1)^\epsilon \sheet_{c_1(\epsilon_1),\dots,c_d(\epsilon_d)}$, where $\epsilon = |\boldsymbol{\epsilon}|=\epsilon_1+\dots+\epsilon_d$ and $c_k(\epsilon_k)= t_k$ if $\epsilon_k=0$, $s_k$ otherwise.
\item For any $H\in (0,1/2]$, the multiparameter fractional Brownian motion is the Gaussian process with covariance given by:
\begin{equation}\label{eq:cov mpfBm}
\EE\left(\mathbb{B}^H_s \mathbb{B}^H_t\right) = \frac{1}{2} \left(\lambda_d\left([0,s]\right)^{2H}+\lambda_d\left([0,t]\right)^{2H}-\lambda_d\left([0,s]\bigtriangleup [0,t]\right)^{2H}\right)\ ,\ s,t\in \R_+^d \ .
\end{equation}
Its extension to an $L^2(\R_+^d,\lambda)$-indexed process which is $L^2$-increment stationary is straightforward from (\ref{eq:covariance}) and has been studied in \cite{Richard}. Hence it is also increment stationary in the wide sense. When only observed as a multiparameter process, it possesses the following increment stationarity property (see \cite{ehemMpfBm}): for any $t\preccurlyeq t'$ and any $\tau \in \R_+^d$,
\begin{equation}\label{eq:measure stationarity}
\lambda\left([0,t']\setminus[0,t]\right) = \lambda([0,\tau])\ \Rightarrow \ \mathbb{B}^H_{t'}-\mathbb{B}^H_{t} \overset{(d)}{=} \mathbb{B}^H_{\tau}\ .
\end{equation}
This is in fact a weak form of the \emph{measure increment stationarity} presented in Section \ref{sec:3}.\\
When $H=1$ and $\mathbf{H}=(\frac{1}{2},\dots,\frac{1}{2})$, $\mathbb{B}^H$ and $\mathbb{W}^{\mathbf{H}}$ above are the same process, known as Brownian sheet. 
\end{enumerate}
\end{example}

One of our initial motivations for this work was to obtain a spectral representation theorem for processes having the measure increment stationarity, and a fractal characterization of the multiparameter fBm based on this stationarity property. We will see that measure increment stationary may be in fact too weak for these purposes, which is why we introduced wide-sense increment stationary.

We only presented Gaussian examples but stable process could also be exhibited (\cite{Samorodnitsky}). These were examples of processes that are compatible with set-indexing and that extend naturally to a function space indexing. If no such natural extension is available, one can always resort to the following result, but with consequences that are discussed at the end of this paragraph and paragraph \ref{subsec:examples}.

\begin{proposition}\label{prop:linearWSIS}
Let $(T,m)$ satisfy the aforementioned conditions. Any second order $T$-indexed process with covariance $C$ extends to a linear $H(C)$-indexed process and thus is wide-sense increment-stationary.
\end{proposition}

\begin{proof}
Let $C$ denote the covariance of $X$ and let $\{C(t_n,\cdot),\ n\in \N\}$ be a basis of $H(C)$ (where $t_n\in\R^d, \forall n$). Then define $\widehat{X}(C(t_n,\cdot)) = X_{t_n}$ for any $n$ and extend $\widehat{X}$ to $H(C)$ by linearity. That is for $\mathtt{f} = \sum_{n=1}^\infty \alpha_n(\mathtt{f}) C(t_n,\cdot)$, where only a finite number of terms in the sum are non-zero, $\widehat{X}(\mathtt{f}) = \sum_{n=1}^\infty \alpha_n(\mathtt{f}) X_{t_n}$. The linearity of $\widehat{X}$ automatically yields the $L^2$-increment stationarity.
\end{proof}
\noindent This result is only here to emphasize how general our definition of increment stationarity is. In fact, having in hands a \emph{linear} $H(C)$-indexed process might not be very useful (at least for the applications we have in mind), since it yields a somehow degenerate spectral decomposition, as we will see in the next section. However this linear process can be considered as a stochastic integral against $X$, whose space of (deterministic) integrands coincides with the RKHS of $X$.

\subsection{Spectral representation theorem for $L^2$-increment stationary processes}

In this section, no particular property of $L^2(T,m)$ is used except that it is a separable Hilbert space. Hence, the stochastic processes that appear here can be indexed by any separable Hilbert space $\Hc$.
\begin{definition}
Similarly to the definition of $L^2$-increment stationarity given in Definition \ref{def:WSIS}, a real-valued process $\{X(h),\ h\in \Hc\}$ is $L^2$-stationary if it has finite second moments at any point and if it satisfies, for any $f,g$ and $h\in \Hc$:
\begin{equation*}
\EE\left(X(f+h) X(g+h)\right) = \EE\left(X(f) X(g)\right) \ .
\end{equation*}
\end{definition}

In the sequel, $E^*\subset \Hc \subset E$ is a Gel'fand triple as in Lemma \ref{lem:HS}, and $S$ denotes the canonical injection from $E^*$ to $\Hc$.

\begin{proposition}\label{prop:spectral}
Let $C:\Hc\times \Hc\rightarrow \R$ be a covariance of the form $C(\kappa,\kappa') = \frac{1}{2}\left(\Phi(\kappa) + \Phi(\kappa') - \Phi(\kappa-\kappa')\right)$ for some symmetric continuous function $\Phi$. Then there exist a non-negative symmetric operator $R:E^*\rightarrow E$, and a finite Borel measure $m$ on $E$ such that:
\begin{align}\label{eq:decomp}
\forall \xi,\quad \Phi(S\xi) = \langle\xi,R\xi\rangle + 2\int_E \frac{1-\cos\langle\xi,x\rangle}{1\wedge \|x\|^2}\ m(\dd x)\ .
\end{align}
Besides, $R\circ i_E$ is a trace-class operator on $E$ (where $i_E = S^*\circ S$ is the canonical injection of $E^*\rightarrow E$), and $m(\{0\})=0$. (Note that the norm appearing in the above integral is the norm of $E$, and we do not precise it in the sequel unless the context is unclear.)
\end{proposition}

\begin{proof}
Due to the form of $C$, the application $\xi\in E^* \mapsto \Phi(S \xi)$ is continuous and negative definite (see Definition 4.3 and Proposition 4.4 in \cite{Bernstein}). Thus, according to Schoenberg's theorem, $\xi\mapsto \exp(-t\Phi(S \xi))$ is positive definite for any $t\in \R_+^*$. The rest of the proof is explained with full details in \cite{Rockner2011}, but we give the main ingredients for the sake of completeness.\\
It follows from Lemma \ref{lem:HS} and Sazonov's theorem (see \cite[Theorem 3.2]{Yan}), according to which a Hilbert-Schmidt map is radonifying, that since $\kappa\mapsto \exp(-\frac{t}{2}\Phi(\kappa))$ is continuous on $H$ for each $t>0$, it is the Fourier transform of a measure $\nu_t$ on $E$, i.e:
\begin{align*}
\forall \xi\in E^*,\quad e^{-\frac{t}{2}\Phi(S\xi)} = \int_E e^{i\langle \xi,x\rangle} \nu_t(\dd x)\ .
\end{align*}
By Lévy's continuity theorem in Hilbert spaces (\cite{Rockner2011}), $\{\nu_t,\ t>0\}$ weakly converges as $t\rightarrow 0$ to the Dirac mass $\delta_0$. Hence $\xi\in E^*\mapsto\exp\left(-\frac{1}{2}\Phi(S\xi)\right)$ is the characteristic function of the infinitely divisible distribution $\nu_1$. So by the Lévy-Khintchine theorem \cite[Theorem VI.4.10]{Parthasarathy}:
\begin{align*}
\forall \xi\in E^*,\quad \Phi(S\xi) = 2i\langle\xi,b\rangle + \langle\xi,R\xi\rangle - 2\int_E\left(e^{i\langle \xi,x\rangle} -1 -\frac{i\langle \xi, x\rangle}{1+\|x\|^2}\right) m_0(\dd x)\ ,
\end{align*}
where $R$ satisfies the hypotheses stated in the proposition, and $m_0$ is a Lévy measure, in the sense that $m_0(\{0\})=0$ and $\int_E \left(1\wedge \|x\|^2\right)\ m_0(\dd x) <\infty$.
Using the equality $\Phi(-\xi) = \Phi(\xi)$, we obtain that for any $\xi\in E^*$:
\begin{equation*}
\langle \xi, b\rangle = \int_E\left(\frac{\langle\xi,x\rangle}{1+\|x\|^2} - \sin\langle\xi,x\rangle\right)\ m_0(\dd x)\ .
\end{equation*}
The linearity in the left hand side of the previous equality implies that for any $n\in\N$, $\int_E \sin\langle n\xi,x\rangle\ m_0(\dd x) = n\int_E \sin\langle\xi,x\rangle\ m_0(\dd x)$. Hence $\int_E \sin\langle\xi,x\rangle\ m_0(\dd x) = 0$ for any $\xi\in E^*$ and it follows that $m_0$ is a symmetric measure. Thus $b=\int_E \frac{x}{1+\|x\|^2}\ m_0(\dd x) = 0$ also. The result follows by defining $m(\dd x) = (1\wedge \|x\|^2)\ m_0(\dd x)$.
\end{proof}

\begin{theorem}\label{th:00}
Let $Y$ be a real-valued $\Hc$-indexed $L^2$-increment stationary process with continuous covariance, and let $E^*\subset \Hc\subset E$ be a Gel'fand triple with Hilbert-Schmidt embedding. Then there exist a symmetric random measure $M$ on $E$, and an uncorrelated random vector $Z$ with covariance operator $R:E\rightarrow E$, such that:
\begin{equation*}
\forall \xi\in E^*,\quad Y_\xi = \int_E \frac{e^{i\langle \xi,x\rangle}-1}{1\wedge \|x\|}\ M(\dd x) + \langle \xi,Z\rangle\ .
\end{equation*}
The previous decomposition extends to $\mathcal{H}$ in the following manner: there exists a linear mapping $\mathcal{Z}:\mathcal{H}\rightarrow L^2(\Omega)$ which is uncorrelated with $M$, such that $\EE\left(\mathcal{Z}(\kappa)^2\right) = (\kappa,\tilde{R} \kappa)_\mathcal{H}$ where $\tilde{R}$ is a symmetric non-negative operator on $\mathcal{H}$ and
\begin{align*}
\forall \kappa\in \mathcal{H},\quad Y(\kappa) = \int_E \gamma(\kappa,x)\ M(\dd x) + \mathcal{Z}(\kappa)\ ,
\end{align*}
where $\gamma$ is the uniformly continuous extension of the mapping $\xi\in E^* \mapsto \frac{1-e^{i\langle\xi,\cdot\rangle}}{1\wedge\|\cdot\|}\in L^2(m)$ to a mapping from $\Hc \rightarrow L^2(m)$.\\
Conversely, any $\Hc$-indexed process with this representation is $L^2$-increment stationary.
\end{theorem}

\begin{proof}
This proof is carried out in two steps. In the first one, we prove the decomposition on $E^*$, while in the second step, we extend it to $\mathcal{H}$.

\emph{First Step.} The $L^2$-increment stationarity implies that the covariance of $Y$ is of the form given in Proposition \ref{prop:spectral} (with a continuous function $\Phi$), thus we let $m$ and $R$ be defined according to the result of this proposition.
For some non-zero $\xi_0\in E^*$, let $X$ be defined by $X_\xi = Y_{\xi+\xi_0} - Y_\xi$. Then $X$ is $L^2$-stationary and its covariance satisfies:
\begin{align*}
\EE\left(X_\xi\ X_\eta\right) &= \EE\left((Y_{\xi+\xi_0}-Y_\xi)\ (Y_{\eta+\xi_0}-Y_\eta)\right)\\
&= \frac{1}{2}\left(\Phi(\xi-\eta+\xi_0) + \Phi(\xi-\eta-\xi_0) - 2\Phi(\xi-\eta)\right)
\end{align*}
and one can check that this quantity can be written $\Psi(\xi-\eta)$ (we omit the dependence in $\xi_0$ in this notation), where $\Psi$ reads:
\begin{equation*}
\forall \xi\in E^*,\quad \Psi(\xi) = \langle\xi_0,R\xi_0\rangle + 2\int_E e^{i\langle\xi,x\rangle}\frac{1-\cos\langle\xi_0,x\rangle}{1\wedge \|x\|^2}\ m(\dd x)\ .
\end{equation*}
Let us define a new finite measure on the Borel sets of $E$ by:
\begin{align*}
\tilde{m}_{\xi_0}(\dd x) = 2\frac{1-\cos\langle\xi_0,x\rangle}{1\wedge \|x\|^2} \mathbf{1}_{\{x\neq 0\}} m(\dd x) + \mathbf{1}_{\{x=0\}} \langle\xi_0,R\xi_0\rangle\ ,
\end{align*}
so that $\Psi$ can be written $\Psi(\xi) = \int_E e^{i\langle \xi,x\rangle}\ \tilde{m}_{\xi_0}(\dd x)$.\\

\noindent We shall now define a process $T_{\xi_0}$ on the vector space $\text{Span}\{e^{i\langle \xi,\cdot\rangle},\ \xi\in E^*\}$ satisfying the following linearity properties: for any $\lambda \in \R,\ \xi,\eta\in E^*$,
\begin{align*}
& T_{\xi_0}\left(\lambda e^{i\langle \xi,\cdot\rangle}\right) =\lambda X_\xi \\
& T_{\xi_0}\left(e^{i\langle \xi,\cdot\rangle} + e^{i\langle \eta,\cdot\rangle}\right) = X_\xi + X_{\eta}\ .
\end{align*}
We claim that this process is well-defined, as there does not exist either couples $(\lambda,\xi)\neq (\lambda',\xi') \in (\R\setminus\{0\})\times E^*$ such that $\lambda e^{i\langle \xi,\cdot\rangle}= \lambda' e^{i\langle \xi',\cdot\rangle}$, nor does there exist couples $(\xi,\eta)\neq (\xi',\eta')\in E^*\times E^*$ such that $e^{i\langle \xi,\cdot\rangle}+e^{i\langle \eta,\cdot\rangle}=e^{i\langle \xi',\cdot\rangle}+e^{i\langle \eta',\cdot\rangle}$.
Note that this process is an isometry of $L^2_{\C}(\tilde{m}_{\xi_0})\rightarrow L^2_\R(\Omega)$ since:
\begin{equation*}
\EE\left(T_{\xi_0}(e^{i\langle \xi,\cdot\rangle})^2\right) = \Psi(0) = \|e^{i\langle \xi,\cdot\rangle}\|^2_{L^2(\tilde{m}_{\xi_0})} \ .
\end{equation*}
Since the vector space spanned by the functions $e^{i\langle\xi,\cdot\rangle},\ \xi\in E^*$ is dense in $L^2(\tilde{m}_{\xi_0})$ (see the following Lemma \ref{lem:density}), we are able to define the following random measure:
\begin{equation*}
\tilde{M}_{\xi_0}(A) = T_{\xi_0}(\mathbf{1}_A)\ ,\quad \forall A\in\mathcal{B}(E)\ ,
\end{equation*}
so that $\tilde{M}_{\xi_0}$ has control measure $\tilde{m}_{\xi_0}$: $\EE\left(\tilde{M}_{\xi_0}(A)\right)=0$ and $\EE\left(\tilde{M}_{\xi_0}(A)\ \tilde{M}_{\xi_0}(B)\right) = \tilde{m}_{\xi_0}(A\cap B)$, for all $A,B\in \mathcal{B}(E)$.
One can now construct a stochastic integral against $\tilde{M}_{\xi_0}$ which satisfies, for any $f\in L^2(\tilde{m}_{\xi_0})$:
\begin{equation*}
\int_E f(x)\ \tilde{M}_{\xi_0}(\dd x) = T_{\xi_0}(f)\ .
\end{equation*}
In particular, for $f=e^{i\langle \xi,\cdot\rangle}$, we recover:
\begin{equation*}
X_\xi^{(\xi_0)} = \int_E e^{i\langle \xi,x\rangle} \ \tilde{M}_{\xi_0}(\dd x)\ ,\quad \forall \xi\in E^*\ .
\end{equation*}
Note that we shall use the notation $X_\xi^{(\xi_0)}$ for $X_\xi$ in the rest of this proof.
By the same density argument as above, there is a random variable $Z_{\xi_0}$ in the $L^2(\Omega)$-closure of $\text{Span}\{X_\xi,\ \xi\in E^*\}$ such that $Z_{\xi_0} = \tilde{M}_{\xi_0}(\{0\})$. At the end of this proof, we will give more details on $Z_{\xi_0}$. In between, let us define the random measure $\underbar{M}_{\xi_0}$ by:
\begin{equation*}
\forall A\in \mathcal{B}(E),\quad \underbar{M}_{\xi_0}(A) = \tilde{M}_{\xi_0}(A) - \mathbf{1}_{\{A\cap \{0\}\neq \emptyset\}} \tilde{M}_{\xi_0}(\{0\})\ ,
\end{equation*}
and the process $\underbar{X}^{(\xi_0)}$ by:
\begin{equation*}
\forall \xi \in E^*,\quad \underbar{X}^{(\xi_0)}_\xi = X^{(\xi_0)}_\xi - Z_{\xi_0} = \int_E e^{i\langle \xi,x\rangle} \underbar{M}_{\xi_0}(\dd x)\ .
\end{equation*}
A few facts can be easily deduced from the previous definitions: firstly, the control measure of $\underbar{M}_{\xi_0}$ is: $$\underbar{m}_{\xi_0} = 2\frac{1-\cos\langle\xi_0,x\rangle}{1\wedge \|x\|^2} \mathbf{1}_{\{x\neq 0\}} m(\dd x)\ ;$$ secondly, $\underbar{X}^{(\xi_0)}$ is still a stationary process; and finally, for any $\xi\in E^*$, $Z_{\xi_0}$ and $\underbar{X}_\xi^{(\xi_0)}$ are uncorrelated.

Let us come back to $X$ and let $\xi_0'\in E^*$: observe that for any $\xi\in E^*$,
\begin{align*}
X_\xi^{(\xi_0+\xi_0')} = X_{\xi+\xi_0}^{(\xi_0')} + X_\xi^{(\xi_0)}\ .
\end{align*}
Thus for any $\xi\in E^*$, $\int_E e^{i\langle \xi,x\rangle} \tilde{M}_{\xi_0+\xi_0'}(\dd x) = \int_E e^{i\langle \xi+\xi_0,x\rangle} \tilde{M}_{\xi_0'}(\dd x) + \int_E e^{i\langle \xi,x\rangle} \tilde{M}_{\xi_0}(\dd x)$. By symmetry, this implies:
\begin{align*}
\forall \xi\in E^*,\quad \int_E e^{i\langle\xi,x\rangle}\ \left(e^{i\langle \xi_0',x\rangle}-1\right)\ \tilde{M}_{\xi_0}(\dd x)=\int_E e^{i\langle\xi,x\rangle}\ \left(e^{i\langle \xi_0,x\rangle}-1\right)\ \tilde{M}_{\xi_0'}(\dd x)\ ,
\end{align*}
which can be transposed to $\underbar{M}$, since the previous integrals cannot charge $\{0\}$:
\begin{align}\label{eq:00}
\forall \xi\in E^*,\quad \int_E e^{i\langle\xi,x\rangle}\ \left(e^{i\langle \xi_0',x\rangle}-1\right)\ \underbar{M}_{\xi_0}(\dd x)=\int_E e^{i\langle\xi,x\rangle}\ \left(e^{i\langle \xi_0,x\rangle}-1\right)\ \underbar{M}_{\xi_0'}(\dd x)\ .
\end{align}
For the finite Borel measure $\underbar{m}_{\xi_0,\xi_0'}(\dd x) := 2(1-\cos\langle\xi_0',x\rangle)\ \underbar{m}_{\xi_0}(\dd x)$, we define for any $A\in \mathcal{B}(E)$ the mapping $\varphi_{\xi_0,\xi_0',A}: x\in E\mapsto \mathbf{1}_A(x) (1\wedge \|x\|) \left((e^{i\langle \xi_0,x\rangle}-1)(e^{i\langle \xi_0',x\rangle}-1)\right)^{-1}$. Since 
\begin{align*}
\EE\left(\left| \int_E \varphi_{\xi_0,\xi_0',A}(x)\ \left(e^{i\langle \xi_0',x\rangle}-1\right) \underbar{M}_{\xi_0}(\dd x) \right|^2\right) = \int_A m(\dd x) <\infty\ ,
\end{align*}
$\varphi_{\xi_0,\xi_0',A} \in L^2(\underbar{m}_{\xi_0,\xi_0'})$ and Lemma \ref{lem:density} states that $\varphi_{\xi_0,\xi_0',A}$ can be approximated by elements in $\text{Span}\{e^{i\langle\xi,\cdot\rangle},\ \xi\in E^*\}$.
Thus Equation (\ref{eq:00}) yields that for any $A\in \mathcal{B}(E)$ such that $A\cap \{0\}=\emptyset$, $\int_A \left(1\wedge \|x\|\right)\ \left(e^{i\langle \xi_0,x\rangle}-1\right)^{-1} \underbar{M}_{\xi_0}(\dd x)$ is independent of $\xi_0$ (and by definition, $\underbar{M}_{\xi_0}(\{0\})=0$). Thus we call this quantity $M(A)$, and one can verify that $M$ is a random measure whose control measure is precisely $m$. From the equality:
\begin{equation*}
\int_E \frac{e^{i\langle \xi,x\rangle} -1}{1\wedge \|x\|} \ M(\dd x) = \int_E \frac{e^{i\langle \xi,x\rangle} -1}{1\wedge \|x\|} \ \frac{1\wedge \|x\|}{e^{i\langle \xi,x\rangle} -1}\ \underbar{M}_\xi(\dd x) = \underbar{M}_\xi(E) \ ,\ \ \forall \xi\in E^*\ ,
\end{equation*}
and due to $\underbar{X}^{(\xi)}_0 = \underbar{M}_\xi(E)$, it is now clear that $Y$ admits the following representation:
\begin{equation}\label{eq:01}
\forall \xi\in E^*,\quad Y_\xi = Z_\xi + \int_E \frac{e^{i\langle \xi,x\rangle} -1}{1\wedge \|x\|} \ M(\dd x)\ .
\end{equation}

\vspace{0.1cm}

To conclude this part of the proof, we need to show that there exists a random variable $Z$ with values in $E$ such that $Z_\xi = \langle \xi,Z\rangle$ and whose covariance operator is $R$. Let us prove that for any $\xi,\eta\in E^*$, $Z_\xi+Z_\eta = Z_{\xi+\eta}$ a.s. Using Equality (\ref{eq:01}),
\begin{align}\label{eq:001}
\EE\left((Z_{\xi+\eta} - Z_\xi - Z_\eta)^2\right)
&= \EE\left(\left| Y_{\xi+\eta} - Y_\xi - Y_\eta + \int_E \frac{e^{i\langle \xi,x\rangle} + e^{i\langle \eta,x\rangle} - e^{i\langle \xi+\eta,x\rangle} - 1}{1\wedge \|x\|}\ M(\dd x) \right|^2\right) \nonumber\\
&= \EE\left(\left|Y_{\xi+\eta} - Y_\xi - Y_\eta\right|^2\right) + \EE\left(\left|\int_E \frac{e^{i\langle \xi,x\rangle} + e^{i\langle \eta,x\rangle} - e^{i\langle \xi+\eta,x\rangle} - 1}{1\wedge \|x\|}\ M(\dd x)\right|^2\right)  \nonumber\\
&\quad\quad + 2\EE\left((Y_{\xi+\eta} - Y_\xi - Y_\eta) \overline{\int_E \frac{e^{i\langle \xi,x\rangle} + e^{i\langle \eta,x\rangle} - e^{i\langle \xi+\eta,x\rangle} - 1}{1\wedge \|x\|}\ M(\dd x)} \right)
\end{align}
We analyse the three summands of the last line separately, and recall that the covariance of $Y$ is given by $C(\xi,\eta) = \frac{1}{2}(\Phi(\xi)+\Phi(\eta)- \Phi(\xi-\eta))$:
\begin{align*}
\EE\left(\left(Y_{\xi+\eta} - Y_\xi - Y_\eta\right)^2\right) &= 2\EE\left(Y_\eta^2\right) - 2\EE\left((Y_{\xi+\eta}-Y_\xi)Y_\eta\right)\\
&= 2\Phi(\xi) + 2\Phi(\eta) - \Phi(\xi+\eta) - \Phi(\xi-\eta)\ .
\end{align*}
The decomposition of $\Phi$ given in (\ref{eq:decomp}) implies that:
\begin{align}\label{eq:02}
\EE\left(\left(Y_{\xi+\eta} - Y_\xi - Y_\eta\right)^2\right) &= 4 \int_E \frac{1-\cos\langle\xi,x\rangle}{1\wedge\|x\|^2}\ m(\dd x) + 4\int_E \frac{1-\cos\langle\eta,x\rangle}{1\wedge\|x\|^2}\ m(\dd x) \nonumber\\
&\quad\quad -2\int_E \frac{1-\cos\langle\xi+\eta,x\rangle}{1\wedge\|x\|^2}\ m(\dd x) -2 \int_E \frac{1-\cos\langle\xi-\eta,x\rangle}{1\wedge\|x\|^2}\ m(\dd x)\ ,
\end{align}
because the quadratic terms annihilate.\\
Next, we remark that $Y_{\xi+\eta} - Y_\xi-Y_\eta = -\int_E \left(e^{i\langle \xi,x\rangle} + e^{i\langle \eta,x\rangle} - e^{i\langle \xi+\eta,x\rangle} - 1\right)\ M(\dd x) + R(\xi,\eta)$, where $R(\xi,\eta)=\tilde{M}_{\xi+\eta}(\{0\}) -\tilde{M}_{\xi}(\{0\})-\tilde{M}_{\eta}(\{0\})$, and also that $\EE\left(\tilde{M}_{\xi}(\{0\})\ \overline{M(A)}\right) = 0$ for any $\xi\in E^*$ and $A\in \mathcal{B}(E)$. Hence $R(\xi,\eta)$ is uncorrelated with $M$, so the sum of the second and third summand in Equation (\ref{eq:001}) is in fact equal to:
\begin{align*}
-\EE\left(\left|\int_E \frac{e^{i\langle \xi,x\rangle} + e^{i\langle \eta,x\rangle} - e^{i\langle \xi+\eta,x\rangle} - 1}{1\wedge \|x\|}\ M(\dd x)\right|^2\right) = - \int_E \frac{\left| e^{i\langle \xi,x\rangle} + e^{i\langle \eta,x\rangle} - e^{i\langle \xi+\eta,x\rangle} - 1 \right|^2}{1\wedge \|x\|^2}\ m(\dd x) \ .
\end{align*}
The sum between this term and (\ref{eq:02}) is precisely $0$. Thus $\EE\left((Z_{\xi+\eta} - Z_\xi - Z_\eta)^2\right)=0$.
We prove similarly that for any $\lambda \in \R$, $Z_{\lambda \xi} = \lambda Z_\xi$ a.s. Hence $Z_\xi$ must take the announced form.

\vspace{0.3cm}

\emph{Second Step.} Let $\Xi(S\xi) = 2\int_E \frac{1-\cos\langle \xi,x\rangle}{1\wedge \|x\|^2}\ m(\dd x)$ be the second part of the covariance $\Phi$. Then $\Xi$ extends to a function on $\mathcal{H}$. Indeed, the mapping:
\begin{align*}
\gamma:\ S(E^*)&\rightarrow L^2(m)\\
S\xi &\mapsto \frac{1-e^{i\langle \xi,\cdot\rangle}}{1\wedge\|\cdot\|}
\end{align*}
satisfies $\|\gamma(S\xi)-\gamma(S\eta)\|_{L^2(m)} = \|\gamma(S(\xi-\eta))\|_{L^2(m)} \leq \Phi\left(S(\xi-\eta)\right)^{1/2}$ for any $\xi,\eta\in E^*$, where the inequality holds since the difference between both terms is precisely $\langle \xi-\eta,R(\xi-\eta)\rangle\geq 0$. Note that $\Phi^{1/2}$ is only a seminorm on $\Hc$ (it might not separate point). Hence we consider the quotient space $E^*/\Phi$ endowed with the proper norm $\Phi^{1/2}$, where the equivalence relation is given by $\xi\sim \eta \Leftrightarrow \Phi\left(S(\xi-\eta)\right)=0$. We still denote by $\gamma$ the previous mapping.
Thus $\gamma$ is uniformly continuous as a mapping from $E^*/\Phi$ to $L^2(m)$. 
Hence by a classical analysis result, it extends to a uniformly continuous mapping (still denoted by $\gamma$) on the completion of $E^*/\Phi$ with respect to the $\Phi^{1/2}$ norm. Since $\Phi$ is continuous in $\Hc$, the closure of $E^*/\Phi$ includes $\mathcal{H}/\Phi$. So $\gamma$ can be finally considered as a mapping on the space $\mathcal{H}/\Phi$. Now define $\tilde{R}$ as follows:
\begin{align*}
\forall \bar{\kappa}\in \Hc/\Phi,\quad \tilde{R}(\bar{\kappa},\bar{\kappa})  = \Phi(\bar{\kappa}) - \|\gamma(\bar{\kappa})\|_{L^2(m)}^2\ ,
\end{align*}
and then $\tilde{R}(\bar{\kappa},\bar{\kappa}')$ by polarization. This is a nonnegative definite symmetric bilinear operator, as the limit of $R$ on $E^*/\Phi$. In fact, $\tilde{R}$ and $\gamma$ are well-defined on $\Hc$ by $\gamma(\kappa) = \gamma(\bar{\kappa})$ and $\tilde{R}(\kappa,\kappa') = \tilde{R}(\bar{\kappa},\bar{\kappa'})$ for any $\kappa,\kappa'\in \Hc$ ($\bar{\kappa}$ denotes the equivalence class of $\kappa$). Indeed if $\kappa_1,\kappa_2$ are two elements in the same equivalence class, $\|\gamma(\kappa_1)- \gamma(\kappa_2)\|_{L^2(m)}\leq \Phi(\kappa_1-\kappa_2)^{1/2}=0$, and:
\begin{align*}
\tilde{R}(\bar{\kappa},\kappa_1) - \tilde{R}(\bar{\kappa},\kappa_2) = \tilde{R}(\bar{\kappa},\bar{0}) &= \frac{1}{2}\left( \Phi(\bar{\kappa}) - \|\gamma(\bar{\kappa})\|_{L^2(m)}^2 +\Phi(\bar{0})- \|\gamma(\bar{0})\|_{L^2(m)}^2 - \Phi(\bar{\kappa}+\bar{0}) + \|\gamma(\bar{\kappa}+\bar{0})\|_{L^2(m)}^2 \right)\\
&=0\ .
\end{align*}

As for the processes, we proceed as follows: define $\{ \mathcal{M}(\kappa)=\int_E \gamma(\kappa)(x)\ M(\dd x),\ \kappa\in \mathcal{H}\}$. This process is well-defined due to the preceding construction of $\gamma$, and it coincides with the process $\int_E \frac{1-e^{i\langle \cdot,x\rangle}}{1\wedge \|x\|} M(\dd x)$ on $E^*$. Then define $\mathcal{Z}(\kappa) = Y(\kappa) - \mathcal{M}(\kappa)$, which coincides with $Z$ if $\kappa\in E^*$. This concludes the proof.
\end{proof}

\begin{lemma}\label{lem:density}
Let $E$ be a separable Banach space and $m$ a finite Borel measure on $E$. Then the space of trigonometric polynomials $S=\text{Span}\left\{e^{i\langle \xi,\cdot\rangle},\ \xi\in E^*\right\}$ is dense in $L^2(E,m)$.
\end{lemma}

\begin{proof}
We first recall that there exists a sequence of trigonometric polynomials on $\R$, let us denote it by $\{T_n: x\in \R \mapsto \sum_{k=1}^{k_n} a_{k,n} e^{i b_{k,n} x} \}_{n\in \N}$, that approximates the identity of $\R$ uniformly on compacts. Thus, for any $\xi\in E^*$, $T_n(\langle \xi, x\rangle)$ converges to $\langle \xi,x\rangle$ for any $x\in E$.\\
Hence, if $\Sigma$ denotes the $\sigma$-algebra generated by $\{e^{i\langle \xi,\cdot\rangle},\ \xi\in E^*\}$, then any $\xi\in E^*$ is $\Sigma$-measurable. Since the $\sigma$-algebra generated by all the $\xi\in E^*$ is the Borel $\sigma$-algebra of $E$, this implies that $\Sigma= \mathcal{B}(E)$.

Now we must adapt the multiplicative system theorem, as given in \cite[Corollary A.2]{Janson}. Let $V$ be the intersection of $\bar{S}$ (the closure of $S$ in $L^2(E,m)$) with the space of bounded Borel-measurable functions on $E$. Then $V$ is closed for bounded convergence, since for any bounded sequence $\{f_n\}_{n\in\N}$ of elements of $V$ that converges pointwise to a bounded function $f$, the dominated convergence theorem implies that $\|f_n-f\|_{L^2(E,m)}\rightarrow 0$. $V$ contains the constant functions and $S$ is closed under multiplication, hence the multiplicative system theorem states that $V$ contains all the bounded $\sigma(S)$-measurable functions, that is all the bounded Borel-measurable functions by our previous remark. Thus $\bar{S}$ contains all the bounded Borel-measurable functions, and this suffices to prove the result.
\end{proof}

\begin{remark}
It is possible to give a similar treatment to $L^2$-stationary processes, in which case the covariance reads:
\begin{align*}
C(S\xi,S\eta) = \Psi(S(\xi-\eta)) = \int_E e^{i\langle\xi-\eta,x\rangle}\ \nu(\dd x)\ ,\ \forall \xi,\eta\in E^*\ ,
\end{align*}
where $\nu$ is a finite Borel measure.
\end{remark}

\subsection{Applications}\label{subsec:examples}

Given a $T$-indexed random field $X$ with covariance $C$, the linear $H(C)$-indexed process $\widehat{X}$ constructed in Proposition \ref{prop:linearWSIS} has the following spectral representation: $\forall f\in H(C),\ \widehat{X}(f) = Z(f)$ where $Z: H(C)\rightarrow L^2(\Omega)$. Hence $\widehat{X}$ has no spectral measure and our theorem does not carry much information in that case. However as we will see in the next example, this does not mean that there is not another process whose restriction is $X$ and which has a spectral measure.

We recall that the covariance of the multiparameter fractional Brownian motion is given in (\ref{eq:cov mpfBm}). On the contrary to the Lévy fBm and the fractional Brownian sheet, the spectral representation for this process is only recent. In \cite{Richard2}, it was obtained as a special case of our theorem, due to special results available for stable measures on Hilbert spaces. Hence the present work yields a more generic and complete (although more lengthy) way to prove that:
\begin{align*}
\forall t\in \R_+^d,\quad \mathbb{B}^H_t = \int_E \gamma\left(\mathbf{1}_{[0,t]},x\right)\ M^H(\dd x) \ , 
\end{align*}
where $E$ is some Hilbert space in which $L^2(\R_+^d)$ is (Hilbert-Schmidt) embedded, $\gamma$ is defined as in Theorem \ref{th:00}, and $M^H$ has control measure $\Delta^H$ which is the Lévy measure of a stable measure on $E$. In particular, this representation has applications on the sample paths of the multiparameter fBm, since $B^H_t$ can now be written as a sum of independent processes if $E$ is sliced into disjoint subsets \cite{Richard2}.\\
It is also interesting to notice that $\Delta^H$ has a similar form to the control measure of the usual fractional Brownian motion. Indeed, we recall the spectral representation of the fractional Brownian motion:
\begin{equation*}
B^H_t = c_H \int_\R \frac{e^{itx}-1}{|x|^{H+\frac{1}{2}}} \mathbb{W}(\dd x)\ ,
\end{equation*}
where $c_H$ is a normalising constant and $\mathbb{W}$ is a complex Gaussian white noise. Hence in that case the control measure is simply $\frac{\lambda(\dd x)}{|x|^{1+2H}}$ while from \cite{Kuelbs73}, we know that $\Delta^H(B) = \int_0^\infty \frac{\dd r}{r^{1+2H}} \int_S \mathbf{1}_{B}(ry)\ \sigma^H(\dd y)$, where $\sigma^H$ is a finite, rotationally invariant measure on the unit sphere $S$ of $E$.

\section{Stationarity and self-similarity characterization}\label{sec:3}

Let us recall that the $L^2(T,m)$-fractional Brownian motion is the centred Gaussian process with covariance (\ref{eq:covariance}). In this section, the choice of $(T,m)$ is unimportant, hence the notation $L^2(T,m)$ becomes simply $L^2$, and $\|\cdot\|$ always refers to the $L^2(T,m)$ norm. We give two characterizations of the $L^2$-fBm: the first one is very similar to the characterization of the Lévy fBm, while the second one uses a notion of stationarity similar to the one defined for set-indexed processes in \cite{ehem,ehemMpfBm}.

\vspace{0.2cm}

We start with some definitions. Consider the set $\mathcal{G}$, which is the restriction of the general linear group of $L^2$ to bounded linear mappings $\varphi:L^2\rightarrow L^2$ such that:
\begin{equation*}
\forall f,g\in L^2, \quad \|f\| = \|g\| \Rightarrow \|\varphi(f)\| = \|\varphi(g)\| \ .
\end{equation*}
Let $\varrho: \mathcal{G} \rightarrow \R_+$ be the application that maps $\varphi$ to the square of its operator norm. Note that for any $\varphi \in \mathcal{G}$ and any $f\in L^2$, $\|\varphi(f)\| = \sqrt{\varrho(\varphi)} \ \|f\|$, and that $\varrho$ is a group morphism.

We will say that an $L^2$-indexed stochastic process $X$ is:
\begin{itemize}
\item $H$--self-similar, if:
\begin{equation}\label{eq:SS1}
\forall a>0,\quad \{a^{-H}X_{af},\ f\in L^2\}\overset{(d)}{=} \{X_f,\ f\in L^2\} \ ;\tag{SS1}
\end{equation}
\item strongly $H$--self-similar, if: 
\begin{equation}\label{eq:SS2}
\forall \varphi\in \mathcal{G},\quad \{X_{\varphi(f)}, \ f\in L^2\}\overset{(d)}{=} \{\varrho(\varphi)^{H} X_f, \ f\in L^2 \}\ ;\tag{SS2}
\end{equation} 
\item strongly $L^2$-increment stationary, if for any translation or orthogonal transformation $\psi$ of $L^2$:
\begin{equation}\label{eq:SI1}
\{X_{\psi(f)} - X_{\psi(0)},\ f\in L^2\}\overset{(d)}{=} \{X_f - X_0,\ f\in L^2\}\ ;\tag{SI1}
\end{equation}
\item weakly $L^2$-increment stationary, if for any $f_1,\dots, f_n \in L^2$, $g_1,\dots,g_n$ and $h\in L^2$:
\begin{equation}\label{eq:SI2}
\left(X_{f_1+h}-X_{g_1+h},\dots,X_{f_n+h}-X_{g_n+h}\right) \overset{(d)}{=} \left(X_{f_1}-X_{g_1},\dots,X_{f_n}-X_{g_n}\right) \ .\tag{SI2}
\end{equation}
\end{itemize}

\noindent The $L^2$-fBm satisfies all the above properties. (\ref{eq:SS1}) and (\ref{eq:SI1}) are direct analogues of the multiparameter properties presented in the introduction. They give a similar characterization:

\begin{proposition}\label{prop:miscCharac1}
Let $X$ be an $L^2$-indexed Gaussian process. $X$ is an $L^2$-fBm if and only if it is $H$--self-similar and increment-stationary in the strong sense (i.e. it satisfies (\ref{eq:SS1}) and (\ref{eq:SI1})), up to normalization of its variance.
\end{proposition}

\begin{proof}
This proof is rather standard compared to the similar characterization of the usual fractional Brownian motion. The first step is to prove that $X$ has mean $0$.  By self-similarity, it is clear that $X_0 = 0$. Let $f_0$ be a unit vector of $L^2$, and any $f,g\in L^2$,
\begin{align*}
\EE\left(X_{f+g} - X_g\right) = \EE\left(X_f - X_0\right) = \EE(X_f) &= \EE\left(X_{\|f\| f_0}\right) = \|f\|^H\ \EE(X_{f_0}) \ ,
\end{align*}
where the first equality is (\ref{eq:SI1}) for a translation, the third is (\ref{eq:SI1}) for an orthogonal transformation mapping $f$ to $\|f\| f_0$, and the last equality is (\ref{eq:SS1}). But self-similarity and rotation invariance also yield:
\begin{equation*}
\EE\left(X_{f+g} - X_g\right) = \left(\|f+g\|^H - \|f\|^H\right)\ \EE(X_{f_0}) \ .
\end{equation*}
The equality between the last two equations implies that $\EE(X_{f_0})=0$, and so $\EE(X_f)=0,\ {\forall f\in L^2}$.

\noindent The covariance follows with the same arguments:
\begin{align*}
\EE\left(X_f - X_g\right)^2 = \EE\left(X_{f-g}\right)^2 = \|f-g\|^{2H}\ \EE(X_{f_0})^2 \ .
\end{align*}
The $L^2$-fBm is called standard if $\EE(X_{f_0})^2=1$ for any unit vector.
\end{proof}

Before stating our second characterization theorem, note that property (\ref{eq:SI2}) is equivalent to $L^2$-increment stationarity defined in Section \ref{sec:2} if $X$ is a Gaussian process. We briefly discuss (\ref{eq:SI2}) and (\ref{eq:SS2}) for $T$-indexed processes which are compatible with set-indexing. So let $X$ be such process, $\widehat{X}$ be its $L^2(T,m)$-indexed extension and $A$ be the associated set-valued mapping. The definition of \emph{measure increment stationarity} (presented in a weak form in (\ref{eq:measure stationarity})) is made precise here, in a form suited to non-Gaussian processes: for any $n\in \N$, any $t_0,t_1,\dots,t_n \in T$, and any $\tau_1,\dots,\tau_n \in T$,
\begin{equation*}
\forall i,j,\ m\left((A_{t_i}\bigtriangleup A_{t_0})\cap (A_{t_j}\bigtriangleup A_{t_0})\right) = m\left(A_{\tau_i}\cap A_{\tau_j}\right) \Rightarrow \left(X_{t_1}-X_{t_0}, \dots, X_{t_n}-X_{t_0}\right) \overset{(d)}{=} \left(X_{\tau_1},\dots, X_{\tau_n}\right)\ .
\end{equation*}
If $X$ is a process such that $\widehat{X}$ satisfies properties (\ref{eq:SI2}) and (\ref{eq:SS2}), then $X$ has the measure increment stationarity. Note that the property (\ref{eq:SS2}) is a generalization of the self-similarity proposed in \cite{ehem}, initially introduced for set-indexed processes.

\begin{proposition}\label{prop:miscCharac2}
Let $X$ be an $L^2$-indexed Gaussian process. $X$ is an $L^2$-fractional Brownian motion of parameter $H\in(0,1)$ if and only if $X$ satisfies (\ref{eq:SI2}) and (\ref{eq:SS2}) of order $H$, up to normalization of its variance.
\end{proposition}

\begin{proof}
We first prove that $X$ is a centred process. Let $f_0\in L^2$ be a unit vector, and for any $f,g \in L^2$ we have:
\begin{align*}
\EE\left(X_{f+g} - X_g\right) = \EE\left(\varrho(\varphi_1)^H X_{f_0} - \varrho({\varphi_2})^H X_{f_0}\right)
\end{align*}
where $\varphi_1,\varphi_2\in \mathcal{G}$ are such that $f+g=\varphi_1(f_0)$ and $g =\varphi_2(f_0)$. We also have, by (\ref{eq:SI2}), that:
\begin{equation*}
\EE\left(X_{f+g} - X_g\right) = \EE\left(X_{f}\right) = \varrho({\varphi_3})^H \EE(X_{f_0})
\end{equation*}
where $\varphi_3 \in \mathcal{G}$ is such that $f = \varphi_3(f_0)$. We know by definition of $\varrho$ that $\varrho(\varphi_1) = \|f+g\|^2$, $\varrho(\varphi_2) = \|g\|^2$ and $\varrho(\varphi_3) = \|f\|^2$. Hence, the equality between the last two equations implies that:
\begin{equation*}
\left(\|f+g\|^{2H} - \|g\|^{2H}\right) \EE\left(X_{f_0}\right) = \|f\|^{2H}\ \EE\left(X_{f_0}\right) \ .
\end{equation*}
Since this is true for any $f,g\in L^2$, we must have $\EE(X_{f_0}) = 0$, and so $\EE(X_f) = 0,\ \forall f\in L^2$.
To obtain the covariance, just notice that by using (\ref{eq:SI2}) and (\ref{eq:SS2}) in the same fashion:
\begin{equation*}
\EE\left((X_f - X_g)^2\right) = \|f-g\|^{2H} \ \frac{\EE\left(X_{f_0}^2\right)}{\|f_0\|^{2H}} = \|f-g\|^{2H}  \ \EE\left(X_{f_0}^2\right)\ .
\end{equation*}
Therefore,
\begin{align*}
\EE\left(X_f\ X_g\right) &= \frac{1}{2}\left(\EE(X_f^2) + \EE(X_g^2) - \EE\left((X_f - X_g)^2\right)\right) \\
&= \frac{1}{2}\EE\left(X_{f_0}^2\right)\left(\|f\|^{2H} + \|g\|^{2H} - \|f - g\|^{2H}\right)\ .
\end{align*}
Finally, stationarity implies that $\EE\left(X_{f_0}^2\right) = \EE\left(X_{g_0}^2\right)$ for any $g_0$ of norm $1$.
\end{proof}

As a final remark, let us observe that we could not prove any such fractal characterization for the multiparameter fractional Brownian motion (i.e. the centred Gaussian process with covariance (\ref{eq:cov mpfBm})). Indeed, let us consider the following form of self-similarity: in (\ref{eq:SS2}), for the special choice of mappings $\varphi_a,\ a\in \R_+^*$ defined by
\begin{align*}
\text{for } t_1,t_2\in \R_+^d \text{ and } \mu_1,\ \mu_2\in \R,\quad \varphi_a\left(\mu_1 \mathbf{1}_{[0,t_1]} + \mu_2 \mathbf{1}_{[0,t_2]}\right) = \mu_1 \mathbf{1}_{[0,a t_1]} + \mu_2 \mathbf{1}_{[0,a t_2]}\ ,
\end{align*}
we say that a multiparameter process is $H$--self-similar if $X_{at} = \widehat{X}(\mathbf{1}_{[0,at]}) \overset{(d)}{=} \varrho(\varphi_a)^H X_t$. Note that here, ${\varrho(\varphi_a) = a^d}$. Despite that $\mathbb{B}^H$ is a process compatible with set-indexing (with $A_t = [0,t]$), that it is measure increment stationary and $H$--self-similar, we do not know if a centred Gaussian process $X$ with these three properties is a multiparameter fractional Brownian motion. If one was willing to use Proposition \ref{prop:miscCharac2} to prove this, the main difficulty would be to construct an $L^2$-indexed process extending the definition of $X$, which we leave as an open problem.

\vspace{0.3cm}

\paragraph{Acknowledgements.} The author thanks Erick Herbin and Ely Merzbach for initiating the discussion on the topics covered in this paper. He is also grateful to Nate Eldredge for his advice on the proof of Lemma \ref{lem:density}.


\begin{thebibliography}{19}
\providecommand{\natexlab}[1]{#1}
\providecommand{\url}[1]{\texttt{#1}}
\expandafter\ifx\csname urlstyle\endcsname\relax
  \providecommand{\doi}[1]{doi: #1}\else
  \providecommand{\doi}{doi: \begingroup \urlstyle{rm}\Url}\fi

\bibitem[Balan et~al.(2015)Balan, Jolis, and Quer-Sardanyons]{BalanJolisQS}
R.~Balan, M.~Jolis, and L.~Quer-Sardanyons.
\newblock {SPDEs with fractional noise in space with index $ H< 1/2$}.
\newblock \emph{Electron. J. Probab.}, pages 1--36, 2015.

\bibitem[Basse-O'Connor et~al.(2012)Basse-O'Connor, Graversen, and
  Pedersen]{BasseOConnor}
A.~Basse-O'Connor, S.-E. Graversen, and J.~Pedersen.
\newblock {Multiparameter processes with stationary increments: Spectral
  representation and integration}.
\newblock \emph{Electron. J. Probab.}, 17, 2012.

\bibitem[Beznea et~al.(2011)Beznea, Cornea, and R\"{o}ckner]{Rockner2011}
L.~Beznea, A.~Cornea, and M.~R\"{o}ckner.
\newblock {Potential theory of infinite dimensional L\'{e}vy processes}.
\newblock \emph{J. Funct. Anal.}, 261\penalty0 (10):\penalty0 2845--2876, 2011.

\bibitem[Dalang et~al.(2009)Dalang, Khoshnevisan, Mueller, Nualart, and
  Xiao]{Dalangminicourse}
R.~Dalang, D.~Khoshnevisan, C.~Mueller, D.~Nualart, and Y.~Xiao.
\newblock \emph{A minicourse on stochastic partial differential equations}.
\newblock Number 1962. Springer, 2009.

\bibitem[Herbin and Merzbach(2006)]{ehem}
E.~Herbin and E.~Merzbach.
\newblock {A Set-indexed fractional Brownian motion}.
\newblock \emph{J. Theoret. Probab.}, 19\penalty0 (2):\penalty0 337--364, 2006.

\bibitem[Herbin and Merzbach(2007)]{ehemMpfBm}
E.~Herbin and E.~Merzbach.
\newblock {The multiparameter fractional Brownian motion}.
\newblock In \emph{Math Everywhere}, pages 93--101. Springer, 2007.

\bibitem[Ito(1954)]{Ito54}
K.~Ito.
\newblock {Stationary random distributions}.
\newblock \emph{Mem. College Sci. Univ. Kyoto Ser. A Math.}, 28\penalty0
  (3):\penalty0 209--223, 1954.

\bibitem[Janson(1997)]{Janson}
S.~Janson.
\newblock \emph{{Gaussian Hilbert spaces}}.
\newblock Cambridge University Press, 1997.

\bibitem[Kuelbs(1973)]{Kuelbs73}
J.~Kuelbs.
\newblock {A representation theorem for symmetric stable processes and stable
  measures on H}.
\newblock \emph{Z. Wahrsch. Verw. Gebiete}, 26\penalty0 (4):\penalty0 259--271,
  1973.

\bibitem[Lifshits(1995)]{Lifshits}
M.~A. Lifshits.
\newblock \emph{Gaussian random functions}, volume 322.
\newblock Springer Science \& Business Media, 1995.

\bibitem[Monrad and Rootz\'{e}n(1995)]{Monrad}
D.~Monrad and H.~Rootz\'{e}n.
\newblock {Small values of Gaussian processes and functional laws of the
  iterated logarithm}.
\newblock \emph{Probab. Theory Related Fields}, 101\penalty0 (2):\penalty0
  173--192, 1995.

\bibitem[Parthasarathy(1967)]{Parthasarathy}
K.~R. Parthasarathy.
\newblock \emph{Probability measures on metric spaces}.
\newblock Academic Press, 1967.

\bibitem[Richard(2015{\natexlab{a}})]{Richard}
A.~Richard.
\newblock {A fractional Brownian field indexed by $L^2$ and a varying Hurst
  parameter}.
\newblock \emph{Stochastic Process. Appl.}, 125:\penalty0 1394--1425,
  2015{\natexlab{a}}.

\bibitem[Richard(2015{\natexlab{b}})]{Richard2}
A.~Richard.
\newblock {Some singular sample path properties of a multiparameter fractional
  Brownian motion}.
\newblock \emph{Submitted}, 2015{\natexlab{b}}.
\newblock \href{http://arxiv.org/abs/1410.4430}{ArXiv:1410.4430}.

\bibitem[Samorodnitsky and Taqqu(1994)]{Samorodnitsky}
G.~Samorodnitsky and M.~S. Taqqu.
\newblock \emph{{Stable non-Gaussian random processes}}.
\newblock 1994.

\bibitem[Schilling et~al.(2012)Schilling, Song, and Vondracek]{Bernstein}
R.L. Schilling, R.~Song, and Z.~Vondracek.
\newblock \emph{Bernstein functions: theory and applications}, volume~37.
\newblock Walter de Gruyter, 2012.

\bibitem[Talagrand(1995)]{Talagrand}
M.~Talagrand.
\newblock {Hausdorff measure of trajectories of multiparameter fractional
  Brownian Motion}.
\newblock \emph{Ann. Probab.}, 23\penalty0 (2):\penalty0 767--775, 1995.

\bibitem[Yaglom(1957)]{Yaglom57}
A.~M. Yaglom.
\newblock {Some classes of random fields in $n$-dimensional space, related to
  stationary random processes}.
\newblock \emph{Theory Probab. Appl.}, 2\penalty0 (3):\penalty0 273--320,
  January 1957.

\bibitem[Yan(1989)]{Yan}
J.~A. Yan.
\newblock {Generalizations of Gross' and Minlos' theorems}.
\newblock In \emph{S\'{e}minaire de probabilit\'{e}s XXIII}, pages 395--404,
  1989.

\end{thebibliography}
\end{document}